\documentclass[11pt]{article}
\usepackage{amssymb, amsmath}
\usepackage{epsfig}
\usepackage{subfigure}
\usepackage{color}

\usepackage[T1]{fontenc}
\usepackage{caption}
\usepackage{booktabs}
\usepackage{siunitx}

\newcommand{\refe}[1]{(\ref{#1})}
\newcommand{\dst}{\displaystyle}

\newcommand{\dsp}{\displaystyle}

\newcommand{\half}{\mbox{$\frac12$}}

\newcommand{\bq}{\begin{equation}}
\newcommand{\eq}{\end{equation}}
\newcommand{\ba}{\begin{array}}
\newcommand{\ea}{\end{array}}

\newtheorem{theorem}{Theorem}
\newtheorem{remark}[theorem]{Remark}
\newtheorem{prop}[theorem]{Proposition}
\newtheorem{lemma}[theorem]{Lemma}
\newtheorem{ex}[theorem]{Example}
\newtheorem{corollary}[theorem]{Corollary}
\newenvironment{proof}{\noindent{\bf Proof:}}{\hfill$\square$}

\begin{document}

\title {\bf Recurrence Relations of the Hypergeometric-type functions on the quadratic-type lattices}

\author{Rezan Sev\.{\i}n\.{\i}k Ad\i g{\"{u}}zel}%[rsa]
\date{}
\maketitle

\begin{center}
{\it Atilim University, Department of
Mathematics,  Incek  06836, Ankara, Turkey}\\
{\it e-mail: rezan.adiguzel@atilim.edu.tr}\\
{\it Tel: +90 312 586 8661}
\end{center}

\begin{abstract}
The central idea of this article is to present a systematic approach to construct some recurrence relations for the 
solutions of the second-order linear difference equation of hypergeometric-type 
defined on the quadratic-type lattices. We introduce some recurrence relations 
for such solutions by also considering their applications to polynomials on the quadratic-type lattices.



\end{abstract}

\section{Introduction}

Hypergeometric functions have been studied by many researchers \cite{NSU, sus1, sus2}, with special interest previously in such functions defined on different type lattices such as uniform lattice like linear-type and non-uniform lattices, like quadratic, $q$-linear and $q$-quadratic types. In 1983, these functions were studied by Nikiforov and Uvarov who started from the second-order linear difference equation of hypergeometric-type satisfied by such functions, thereby paving the way for this theory to be developed by several other authors (see e.g. \cite{RAC, RM, ARS, ks, NSU, sus1, sus2}). Discrete polynomials are in the special class of these kind of hypergeometric functions and used in many problems \cite{RR1, RR2, RYR, ks, NSU, sus1, sus2}.

 In particular, $q$-polynomials on the $q$-quadratic lattices have been of particular interest in recent studies
(see e.g. \cite{RR1, RR2, RYR, ks, NSU, sus1, sus2})
since they are the most general discrete orthogonal families, from which all the other 
hypergeometric orthogonal polynomials can be obtained.
Such polynomials are the solutions of the second-order difference equation 
of hypergeometric-type defined on the $q$-quadratic lattices.

In this work, we introduce an approach to construct recurrence relations for the hypergeometric functions on the $q$-quadratic lattices $x(s)=c_1q^s+c_2q^{-s}+c_3$ which, also cover the hypergeometric-type functions
on the quadratic lattices $x(s)=c_1s^2+c_2s+c_3$ as a limit case when $q\longrightarrow 1$ 
(see e.g. \cite{ran1}). Here, we also apply $q$-Racah and dual Hahn polynomials on the $q$-quadratic and quadratic lattices.

Since, in several quantum-mechanical models, the wave functions can be expressed in terms of some  hypergeometric-type functions, such recurrence relations give more information about the physical systems modelled by such functions. In fact, the recurrence relations are more useful for the evaluation of these functions than the direct method (see e.g.,  \cite{RCQ, GST1, GST2} and the references therein).

This paper is motivated by the work done by  R. \'{A}lvarez-Nodarse et al. \cite{ran3, RC1, RC2}. In fact,  in \cite{RC1}, the authors considered the continuous case and obtained some recurrence relations for the Jacobi, Laguerre and Hermite polynomials in addition to the
difference analogues of hypergeometric functions
on the linear lattices $x(s)=s$ to apply the theory 
Hahn, Meixner, Charlier and Kravchuk polynomials. In \cite{RC2}, the authors studied the 
difference analogues of hypergeometric functions
on the linear-type lattices, and later applied the theory to the $q$-polynomials on $q$-linear lattices $x(s)=c_1q^s+c_2$ while considering the 
big $q$-Jacobi, Alternative $q$-Charlier polynomials as applications. For the quadratic case, there are only a few known recurrence relations (see the results by Suslov in \cite{sus1, sus2}). 
As such, the main aim of the present paper is to extend the results of \cite{RC2} to 
the general quadratic-type lattice and develope constructive approach 
for the recurrence relations of such functions.
Here we go further and consider the recurrence relations for the functions on the quadratic-type lattices
and apply the theory to the $q$-Racah and dual Hahn polynomials, thus expanding the results of the papers \cite{ran3, RC1, RC2}.

Notice that since the lattice considered in this paper is not linear-type, 
the general results of \cite{RC2} may not be applied. In particular, some of the representative examples considered in \cite{RC2} cannot be obtained for the quadratic-type lattices.

The structure of the paper is as follows: In section 2, the preliminary results are 
introduced. In section 3 and 4, the general theorems
for recurrence relations are given.
Finally, the last section concludes the paper with some representative examples.

%%%%%%%%%%%%%%%%%%%%%%%%%%%%%%%%%%%%%%%%%%%%%%%%%%%%%%%%%%%%%%%%%%%%%%%%%
\section{Preliminaries}
We include some useful information (see e.g. \cite{ran1, NSU}) on the $q$-hypergeometric functions
 needed for the rest of the paper.

The hypergeometric functions on the non-uniform lattices satisfy 
the following second-order difference equation of hypergeometric-type on the non-uniform lattices
\bq\label{diffeqn}
\sigma(s)\frac{\Delta}{\Delta x(s-\frac 12)}\Big[\frac{\nabla y(s)}{\nabla x(s)}\Big]+\tau(s)\frac{\Delta y(s)}{\Delta x(s)}+\lambda y(s)=0,
\eq
where
\bq\label{c1}
\sigma(s)=\widetilde{\sigma}(x(s))-\frac 12\widetilde{\tau}(x(s))\Delta x(s-\frac 12), \quad \tau(s)=\widetilde{\tau}(x(s)).
\eq
Here, $\Delta y(s)=y(s+1)-y(s)$ and $\nabla y(s)=y(s)-y(s-1)$ are the forward and backward 
difference operators, respectively, where
\bq\label{forwardbackward}
\Delta y(s)=\nabla y(s+1),
\eq
and the coefficients $\widetilde{\sigma}(x(s))$ and $\widetilde{\tau}(x(s))$
are polynomials in $x(s)$ of degree at most 2 and 1, respectively, and $\lambda$ is a constant.

In this paper, we study the quadratic-type lattices: the so-called quadratic lattice
\bq\label{quadraticlattice}
x(s)=c_1s^2+c_2s+c_3,
\eq
and the $q$-quadratic lattice,
\bq\label{q-quadraticlattice}
\!\!\!\!\!\!\!\!\!\!\!\!\!\!\!\!\!x(s)=c_1(q)q^s+c_2(q)q^{-s}+c_3(q)=c_1(q)\Big[q^s+q^{-s-\mu}\Big]+c_3(q), \quad q^{-\mu}=\frac{c_2(q)}{c_1(q)}
\eq
with $c_1\neq 0, c_1(q)\neq 0$.

\begin{remark}
Quadratic-type lattices have the following properties:
\begin{eqnarray}
&&\displaystyle\frac{x(s+k)+x(s)}{2}=\alpha_kx_k(s)+\beta_k,\label{latticeprop1}\\
&&x(s+k)-x(s)=\gamma_k\Delta x_k(s-\frac 12)=\gamma_k\nabla x_k(s+\frac 12)\label{latticeprop2}
\end{eqnarray} 
where
\bq\label{xk}
x_{k}(s)=x(s+\frac{k}{2})
\eq
and
\bq\label{alphabetagammak}
\alpha_k=\frac{q^{\frac k2}+q^{-\frac k2}}{2},\quad \beta_k=-\frac{c_3}{2}\Big(q^{\frac k4}-q^{-\frac k4}\Big)^2,\quad
\gamma_k=[k]_q.
\eq
Here, $[k]_q$ is the symmetric $q$-number defined by
\bq\label{qnumber}
[k]_q=\frac{q^{\frac k2}-q^{-\frac k2}}{q^{\frac 12}-q^{-\frac 12}}.
\eq
\end{remark}

\begin{theorem}
\cite{NSU, sus1}
The difference equation \refe{diffeqn}
has a particular solution 
\bq\label{sumynu}
y_{\nu}(z)=\frac{C_{\nu}}{\rho(z)}\sum_{s=a}^{b-1}\frac{\rho_{\nu}(s)\nabla x_{\nu +1}(s)}{\Big[x_{\nu}(s)-x_{\nu}(z)\Big]^{\nu +1}},
\eq
provided that the condition 
\bq\label{sumynucondition}
\frac{\sigma(s)\rho_{\nu}(s)\nabla x_{\nu +1}(s)}{\Big[x_{\nu-1}(s)-x_{\nu-1}(z+1)\Big]^{\nu +1}}\Bigg|_{s=a}^b=0
\eq
is satisfied. Here, $C_{\nu}$ is a constant.
\end{theorem}

Notice that $\rho(s)$ and $\rho_{\nu}(s)$ satisfy the following Pearson equations
\begin{align}
\frac{\rho(s+1)}{\rho(s)}&=\frac{\sigma(s)+\tau(s)\Delta x(s-\frac 12)}{\sigma(s+1)}=\frac{\phi(s)}{\sigma(s+1)},\nonumber\\
\frac{\rho_{\nu}(s+1)}{\rho_{\nu}(s)}&=\frac{\sigma(s)+\tau_{\nu}(s)\Delta x_{\nu}(s-\frac 12)}{\sigma(s+1)}=\frac{\phi_{\nu}(s)}{\sigma(s+1)},\label{pearsonrhonu}
\end{align}
where
\bq\label{taunu}
\tau_{\nu}(s)=\dsp\frac{\sigma(s+\nu)-\sigma(s)+\tau(s+\nu)\Delta x(s+\nu-\frac 12)}{\Delta x_{\nu -1}(s)},
\eq
and, therefore
\bq\begin{array}{lllll}\label{phinu}
\phi_{\nu}(s)&=\sigma(s)+\tau_{\nu}(s)\Delta x_{\nu}(s-\frac 12)\\
             &=\sigma(s+1)+\tau_{\nu-1}(s+1)\Delta x_{\nu-1}(s+\frac 12)\\
						 &=...\\
             &=\sigma(s+\nu)+\tau(s+\nu)\Delta x(s+\nu-\frac 12).
\end{array}
\eq
Notice that
\bq\label{phinuphi}
\phi_{\nu}(s)=\phi(s+\nu)=\sigma(s+\nu)+\tau(s+\nu)\Delta x(s+\nu-\frac 12),
\eq
where, $\nu\in C$ is the solution of 
\[
\lambda_{\nu}+[\nu]_q\Big\{\alpha_{\nu-1}\widetilde{\tau}'+[\nu-1]_q\frac{\widetilde{\sigma}''}{2}\Big\}=0,
\]
with $[x]_q$ and $\alpha_k$ defined by \refe{qnumber} and \refe{alphabetagammak}, respectively.

In the following, we will use the function $\widetilde{\sigma}_{\nu}(s)$ defined as 
\bq\label{sigmatildenu}
\widetilde{\sigma}_{\nu}(s)=\sigma(s)+\frac 12\tau_{\nu}(s)\Delta x_{\nu}(s-\frac 12).
\eq

By \refe{phinu} and \refe{sigmatildenu},
\begin{eqnarray}
\phi_{\nu}(s)+\sigma(s)&=&2\widetilde{\sigma}_{\nu}(s),\label{phinu+sigma}\\
\phi_{\nu}(s)-\sigma(s)&=&\tau_{\nu}(s)\Delta x_{\nu}(s-\frac 12)\label{phinu-sigma}.
\end{eqnarray}

The generalized power of the lattices $x_m(s)$, given in \refe{xk}, are defined as \cite{ran1}
\begin{eqnarray*} 
&&\Big[x_m(s)-x_m(z)\Big]^{(k)}=\prod_{i=0}^{k-1}(x_m(s)-x_m(z-i)), \quad k\in\mathbb{N}\\
&&\Big[x_m(s)-x_m(z)\Big]^{(0)}=1.
\end{eqnarray*}

The generalized power for the lattices \refe{quadraticlattice} and
\refe{q-quadraticlattice} are obtained as follows:\\
{For the quadratic lattice of the form \refe{quadraticlattice}}
\bq\label{generalizedpowerquadraticlattice}
\Big[x_{\nu}(s)-x_{\nu}(z)\Big]^{(\alpha)}=c_1^{\alpha}\frac{\Gamma(s-z+\alpha)\Gamma(s+z+\nu+\mu+1)}{\Gamma(s-z)\Gamma(s+z+\nu-\alpha+\mu+1)}, \quad \mu=\frac{c_2}{c_1}.
\eq
{ For the $q$-quadratic lattice of the form \refe{q-quadraticlattice}}
\begin{multline}\label{generalizedpowerq-quadraticlattice}
\Big[x_{\nu}(s)-x_{\nu}(z)\Big]^{(\alpha)}=
\frac{\Gamma_q(s-z+\alpha)\Gamma_q(s+z+\nu+C+1)}{\Gamma_q(s-z)\Gamma_q(s+z+\nu-\alpha+C+1)}q^{-\alpha(s+\frac{\nu}{2})}\\
\times\Big[c_1(q)(1-q)^2\Big]^{\alpha}=c_1^{\alpha}(q)q^{-\alpha(s+\frac{\nu}{2})}\dsp\frac{(q^{s-z};q)_{\infty}(\eta q^{s+z+\nu-\alpha+1};q)_{\infty}}{(q^{s-z+\alpha};q)_{\infty}(\eta q^{s+z+\nu+1};q)_{\infty}}, 
\end{multline}
where $C=\frac{\log (c_2(q)/c_1(q))}{\log q}$, $\eta=\frac{c_2(q)}{c_1(q)}$ and
classical $q$-Gamma function, $\Gamma_q$, is related to the infinite $q$-product \cite{GR} by formula
$$
\Gamma_q(s)=(1-q)^{1-s}\frac{(q;q)_\infty}{(q^s;q)_\infty},\quad 0<q<1.
$$
Here, the infinite $q$-product \cite{GR} is defined by $(a;q)_{\infty}=\dsp\prod_{k=0}^{\infty}(1-aq^k)$.

\begin{prop}\label{latticeprop} \cite{ran1, RC2, sus1}
Let $\nu$ be a complex number with $m$, $k$ as positive integers
with $m\geq k$. For the quadratic-type lattice of the form \refe{quadraticlattice} and \refe{q-quadraticlattice}, we have

\begin{eqnarray}
	&&\dsp\frac{\Big[x_{\nu}(s)-x_{\nu}(z)\Big]^{(m)}}{\Big[x_{\nu}(s)-x_{\nu}(z)\Big]^{(k)}}=\Big[x_{\nu}(s)-x_{\nu}(z-k)\Big]^{(m-k)},\label{1}\\
&&\dsp\frac{\Big[x_{\nu}(s)-x_{\nu}(z)\Big]^{(m+1)}}{\Big[x_{\nu-1}(s)-x_{\nu-1}(z)\Big]^{(m)}}=\Big[x_{\nu-m}(s+m)-x_{\nu-m}(z)\Big],\label{2}\\
&&\dsp\frac{\Big[x_{\nu}(s)-x_{\nu}(z)\Big]^{(m+1)}}{\Big[x_{\nu-1}(s+1)-x_{\nu-1}(z)\Big]^{(m)}}=\Big[x_{\nu-m}(s)-x_{\nu-m}(z)\Big].\label{3}	
\end{eqnarray}

\end{prop}

The proof is straightforward  using \refe{generalizedpowerquadraticlattice} and 
\refe{generalizedpowerq-quadraticlattice}, hence it is omitted.
The generalization of the above expressions can be written with the following lemma.

\begin{lemma}\label{latticelemma}
Let $\mu_i$ and $\nu_i$, $i=1,2,3$ be complex numbers such that 
the differences $\nu_i-\nu_j$ and $\mu_i-\mu_j$ $i,j=1,2,3$
are integers and 
\bq
\mu_0-\mu_i\geq \nu_0-\nu_i
\eq
where $\nu_0$ is the $\nu_i$, $i=1,2,3$ with the largest real part, and
$\mu_0$ is the $\mu_i$, $i=1,2,3$ with the largest real part.
Then, the ratio of the generalized power can be calculated with
the following formulas:

\mbox{1. If $\nu_i=\nu_0$}

\[ 
\dsp\frac{\Big[x_{\nu_0}(s)-x_{\nu_0}(z)\Big]^{(\mu_0+1)}}{\Big[x_{\nu_i}(s)-x_{\nu_i}(z)\Big]^{(\mu_i+1)}} =
         \Big[x_{\nu_0}(s)-x_{\nu_0}(z-\mu_i-1)\Big]^{(\mu_0-\mu_i)}.     
				\]
				
\mbox{2. If  $\nu_0-\nu_i>0$ and $\mu_0-\mu_i=n$, $\nu_0-\nu_i=n$}
				\[
				\dsp\frac{\Big[x_{\nu_0}(s)-x_{\nu_0}(z)\Big]^{(\mu_0+1)}}{\Big[x_{\nu_i}(s)-x_{\nu_i}(z)\Big]^{(\mu_i+1)}} =
        \dsp\prod_{i=0}^{n-1}\Big[x_{\nu_0-\mu_0}(s+\mu_0-i)-x_{\nu_0-\mu_0}(z)\Big].
				\]

\mbox{3. If  $\nu_0-\nu_i>0$ and $\mu_0-\mu_i=n$, $\nu_0-\nu_i=n-k$, $(\nu_0-\nu_i<n)$}
\begin{multline*}
				\dsp\frac{\Big[x_{\nu_0}(s)-x_{\nu_0}(z)\Big]^{(\mu_0+1)}}{\Big[x_{\nu_i}(s)-x_{\nu_i}(z)\Big]^{(\mu_i+1)}} =
\dsp\prod_{l=0}^{n-k-1}\Big[x_{\nu_0-\mu_0}(s+\mu_0-l)-x_{\nu_0-\mu_0}(z)\Big]\\
\times\dsp\prod_{j=0}^{k-1}\Big[x_{\nu_i}(s)-x_{\nu_i}(z-\mu_0+n-1-j)\Big]. 
\end{multline*}
\end{lemma}
\begin{proof}
At this stage, we only sketch the proof for the 3rd case, and
the others can be done in an analogous way.
One can write the ratio of the generalized power in the 3rd case by 
\begin{eqnarray*}
				\!\!\!\!\!\!\!\!\!&\!\!\!\!\!\!\!\!\!&\!\!\!\!\!\!\!\!\!\dsp\frac{\Big[x_{\nu_0}(s)-x_{\nu_0}(z)\Big]^{(\mu_0+1)}}{\Big[x_{\nu_i}(s)-x_{\nu_i}(z)\Big]^{(\mu_i+1)}}=
				\frac{\Big[x_{\nu_0}(s)-x_{\nu_0}(z)\Big]^{(\mu_0+1)}}{\Big[x_{\nu_0-1}(s)-x_{\nu_0-1}(z)\Big]^{(\mu_0)}} 
				\frac{\Big[x_{\nu_0-1}(s)-x_{\nu_0-1}(z)\Big]^{(\mu_0)}}{\Big[x_{\nu_0-2}(s)-x_{\nu_0-2}(z)\Big]^{(\mu_0-1)}} \\
											\!\!\!\!\!\!\!\!\!&\!\!\!\!\!\!\!\!\!&\!\!\!\!\!\!\!\!\!\times \frac{\Big[x_{\nu_0-2}(s)-x_{\nu_0-2}(z)\Big]^{(\mu_0-1)}}{\Big[x_{\nu_0-3}(s)-x_{\nu_0-3}(z)\Big]^{(\mu_0-2)}} ...\frac{\Big[x_{\nu_0-n+k+1}(s)-x_{\nu_0-n+k+1}(z)\Big]^{(\mu_0-n+k+2)}}{\Big[x_{\nu_0-n+k}(s)-x_{\nu_0-n+k}(z)\Big]^{(\mu_0-n+k+1)}} \\
											\!\!\!\!\!\!\!\!\!&\!\!\!\!\!\!\!\!\!&\!\!\!\!\!\!\!\!\!\times		\frac{\Big[x_{\nu_0-n+k}(s)-x_{\nu_0-n+k}(z)\Big]^{(\mu_0-n+k+1)}}{\Big[x_{\nu_0-n+k}(s)-x_{\nu_0-n+k}(z)\Big]^{(\mu_0-n+k)}}...
											\frac{\Big[x_{\nu_0-n+k}(s)-x_{\nu_0-n+k}(z)\Big]^{(\mu_0-n+2)}}{\Big[x_{\nu_0-n+k}(s)-x_{\nu_0-n+k}(z)\Big]^{(\mu_0-n+1)}},
\end{eqnarray*}
where $\mu_0-\mu_i=n$, $\nu_0-\nu_i=n-k$. 
Then, from the hypothesis and the formulas \refe{1} and \refe{2} of Proposition \ref{latticeprop},
the results in the lemma follow.

\end{proof}

%%%%%%%%%%%%%%%%%%%%%%%%%%%%%%%%%%%%%%%%%%%%%%%%%%%%%%%%%%%%%%%%%%%%%%%%%%%
%%%%%%%%%%%%%%%%%%%%%%%%%%%%%%%%%%%%%%%%%%%%%%%%%%%%%%%%%%%%%%%%%%%%%%%%%%%

\section{Recurrence relation on the quadratic-type lattices}

Here, we obtain the general recurrence relation for the functions on the 
quadratic-type lattices defined by \refe{quadraticlattice} and \refe{q-quadraticlattice}.
To do so, we  generalize the idea used 
for the linear-type lattices in the recent papers \cite{ran3, RC1,RC2}. Next, we prove the following lemma as a generalization of the linear-type lattices considered in Lemma 3.2. of \cite[page 4]{RC2} for the quadratic-type lattices.

\begin{lemma}\label{relationlemma}
Let $x(z)$ be quadratic-type lattices of the form \refe{quadraticlattice} and \refe{q-quadraticlattice}.
Then, the following linear relation holds
\bq\label{linearrelation}
\sum_{i=1}^3A_i(z)\Psi_{\nu_i,\mu_i}(z)=0,
\eq
where the coefficients $A_i(z)$ are non-zero polynomial functions in $x(z)$ and
\bq\label{phi1}
\Psi_{\nu,\mu}(z)=\sum_{s=a}^{b-1}\frac{\rho_{\nu}(s)\nabla x_{\nu+1}(s)}{\Big[x_{\nu}(s)-x_{\nu}(z)\Big]^{(\mu+1)}}
\eq
provided that the differences 
$\nu_i-\nu_j$ and $\mu_i-\mu_j$ $i,j=1,2,3$
are integers such that $\mu_0-\mu_i\geq \nu_0-\nu_i$, $i=1,2,3,$ and the following condition holds
\bq\label{bc1}
\frac{\sigma(s)\rho_{\nu_*}(s)x^k(s)}{\Big[x_{\nu_0-1}(s)-x_{\nu_0-1}(z)\Big]^{(\mu_0)}}\Bigg|_{s=a}^b=0, \quad k=0,1,2,....
\eq
Where, $\nu_*$ is the $\nu_i$, $i=1,2,3$ with the smallest real part,
$\nu_0$ is the $\nu_i$, $i=1,2,3$ with the largest real part and
$\mu_0$ is the $\mu_i$, $i=1,2,3$ with the largest real part.

\end{lemma}

\begin{proof}
By adding the function $\Psi_{\nu_i,\mu_i}$ defined by \refe{phi1} into the sum,
we have
\[
\begin{array}{ll}
\dsp\sum_{i=1}^3\!\!&A_i(z)\Psi_{\nu_i,\mu_i}(z)=\dsp\sum_{i=1}^3A_i(z)\sum_{s=a}^{b-1}\dsp\frac{\rho_{\nu_i}(s)\nabla x_{\nu_i+1}(s)}{\Big[x_{\nu_i}(s)-x_{\nu_i}(z)\Big]^{(\mu_i+1)}}\\
&=\dsp\sum_{s=a}^{b-1}\sum_{i=1}^3A_i(z)\dsp\frac{\rho_{\nu_i}(s)\nabla x_{\nu_i+1}(s)}{\Big[x_{\nu_i}(s)-x_{\nu_i}(z)\Big]^{(\mu_i+1)}}=\dsp\sum_{s=a}^{b-1}\dsp\frac{\rho_{\nu_*}(s)}{\Big[x_{\nu_0}(s)-x_{\nu_0}(z)\Big]^{(\mu_0+1)}}\times\\
&\dsp\left(\sum_{i=1}^3A_i(z)\frac{\rho_{\nu_i}(s)}{\rho_{\nu_*}(s)}\nabla x_{\nu_i+1}(s)\frac{\Big[x_{\nu_0}(s)-x_{\nu_0}(z)\Big]^{(\mu_0+1)}}{\Big[x_{\nu_i}(s)-x_{\nu_i}(z)\Big]^{(\mu_i+1)}}\right)

\end{array}
\]
where 
\bq\label{rhonu*}
\rho_{\nu_i}(s)=\phi(s+\nu_*)\phi(s+\nu_*+1)...\phi(s+\nu_i-1)\rho_{\nu_*}(s)
\eq
by the Pearson equation \refe{pearsonrhonu}. Thus, we have
\[
\dsp\sum_{i=1}^3A_i(z)\Psi_{\nu_i,\mu_i}(z)=
\dsp\sum_{s=a}^{b-1}\dsp\frac{\rho_{\nu_*}(s)}{\Big[x_{\nu_0}(s)-x_{\nu_0}(z)\Big]^{(\mu_0+1)}}\Pi(s)
\]
where
\begin{eqnarray}\label{pi}
\Pi(s)&=&
\dsp\sum_{i=1}^3A_i(z)\nabla x_{\nu_i+1}(s)\frac{\Big[x_{\nu_0}(s)-x_{\nu_0}(z)\Big]^{(\mu_0+1)}}{\Big[x_{\nu_i}(s)-x_{\nu_i}(z)\Big]^{(\mu_i+1)}}\times\\
&&\phi(s+\nu_*)\phi(s+\nu_*+1)...\phi(s+\nu_i-1)\nonumber
\end{eqnarray}
where the ratio of the generalized power can be computed using Lemma \ref{latticelemma}.

We need to show that there exists a polynomial $Q(s)$ in $z=q^s, s\in\mathbb{Z}$; $\{1,z^{-1},z,z^{-2}, z^2,...\}$ satisfying
\bq\label{Q}
\frac{\rho_{\nu_*}(s)}{\Big[x_{\nu_0}(s)-x_{\nu_0}(z)\Big]^{(\mu_0+1)}}\Pi(s)=\Delta\left[
\frac{\sigma(s)\rho_{\nu_*}(s)}{\Big[x_{\nu_0-1}(s)-x_{\nu_0-1}(z)\Big]^{(\mu_0)}}Q(s)
\right].
\eq
If $Q(s)$ exists, then the sum in $s$ over $s=a$ to $b-1$,
together with the boundary condition \refe{bc1} lead to the relation \refe{linearrelation}.
By substituting the $q$-quadratic lattice $x(s)=c_1(q)q^s+c_2(q)q^{-s}+c_3(q)$ in each factors of $\Pi(s)$ in \refe{pi},
one can rewrite it as a polynomial in $z=q^s$ and $1/z=q^{-s}$, which 
is a special class of Laurent polynomials \cite{JN},
\[
\Lambda_{2n}=\{R\in\Lambda_{-n,n}| \, {\rm the \, coefficent \, of } \, z^n \, {\rm is \, nonzero}\}
\]
whose basis is $\{1,z^{-1},z,z^{-2},z^2,z^{-3},z^3,...\}$, where $z=q^s, s\in\mathbb{Z}$ and its L-degree is $2n.$

In order to prove the existence of the polynomial $Q(s)$, we rewrite the right hand-side of \refe{Q}
\[
\begin{array}{lll}
\!\!\!&\!\!\!\dsp\frac{\sigma(s+1)\rho_{\nu_*}(s+1)}{\Big[x_{\nu_0-1}(s+1)-x_{\nu_0-1}(z)\Big]^{(\mu_0)}}Q(s+1)-
\dsp\frac{\sigma(s)\rho_{\nu_*}(s)}{\Big[x_{\nu_0-1}(s)-x_{\nu_0-1}(z)\Big]^{(\mu_0)}}Q(s)=\\
\!\!\!&\!\!\!\dsp\frac{\rho_{\nu_*}(s)}{\Big[x_{\nu_0}(s)-x_{\nu_0}(z)\Big]^{(\mu_0+1)}}\left(
\sigma(s+1)\dsp\frac{\rho_{\nu_*}(s+1)}{\rho_{\nu_*}(s)}\frac{\Big[x_{\nu_0}(s)-x_{\nu_0}(z)\Big]^{(\mu_0+1)}}{\Big[x_{\nu_0-1}(s+1)-x_{\nu_0-1}(z)\Big]^{(\mu_0)}}\right.\\
\!\!\!&\!\!\!\times Q(s+1)-\left.\sigma(s)\dsp\frac{\Big[x_{\nu_0}(s)-x_{\nu_0}(z)\Big]^{(\mu_0+1)}}{\Big[x_{\nu_0-1}(s)-x_{\nu_0-1}(z)\Big]^{(\mu_0)}}Q(s)
\right).
\end{array}
\]
By using the Pearson equation \refe{pearsonrhonu} and  formulas
\refe{3} and \refe{2} of Proposition \ref{latticeprop}, respectively,
one gets
\begin{eqnarray*}
&&\frac{\rho_{\nu_*}(s)}{\Big[x_{\nu_0}(s)-x_{\nu_0}(z)\Big]^{(\mu_0+1)}}\Pi(s)=\frac{\rho_{\nu_*}(s)}{\Big[x_{\nu_0}(s)-x_{\nu_0}(z)\Big]^{(\mu_0+1)}}\Big\{\phi_{\nu_*}(s)\times\\
&&\Big[x_{\nu_0-\mu_0}(s)
-x_{\nu_0-\mu_0}(z)\Big]Q(s+1)-\sigma(s)\Big[x_{\nu_0-\mu_0}(s+\mu_0)-x_{\nu_0-\mu_0}(z)\Big]Q(s)\Big\}.
\end{eqnarray*}
Therefore,
\begin{eqnarray}\label{piQ}
\Pi(s)&=&\phi_{\nu_*}(s)\Bigg[x_{\nu_0-\mu_0}(s)-x_{\nu_0-\mu_0}(z)\Bigg]Q(s+1)\nonumber\\
&-&\sigma(s)\Bigg[x_{\nu_0-\mu_0}(s+\mu_0)-x_{\nu_0-\mu_0}(z)\Bigg]Q(s),
\end{eqnarray}
where $\phi_{\nu_*}(s)=\sigma(s)+\tau_{\nu_*}(s)\nabla x_{\nu_*+1}(s)$.

Recall that $\Pi(s)$ is a Laurent polynomial on the basis $\{1,z^{-1},z,z^{-2},z^2,z^{-3},$ $z^3,...\}$, where $z=q^s$.
Then,  $\sigma(s)$ is a polynomial of the degree at most two in $x(s)$ and also a Laurent polynomial of L-degree at most four,
whose basis is $\{1,z^{-1},z,z^{-2},z^2\}$, where $z=q^s$.
Moreover, $\tau_{\nu_*}(s)$ is a polynomial of degree one in $x_{\nu_*}(s)$ and also a Laurent polynomial of L-degree two, 
whose basis is $\{1,z^{-1},z\}$, where $z=q^s$. In addition, $x_k(s)$ is a Laurent polynomial of L-degree two,
whose basis is $\{1,z^{-1},z\}$, where $z=q^s$.

Therefore, by substituting the q-quadratic lattice \refe{q-quadraticlattice} and taking into account property \refe{xk}, one can see that $Q(s)$
is also a Laurent polynomial, whose L-degree is at least six less than the L-degree of $\Pi(s)$.

Note that two Laurent polynomials are equal if their coefficients are the same just like the case with the ordinary polynomials.
Then, one can use the equality of 
the coefficents of the Laurent polynomials
in order to find $A_i(z)$. This completes the proof.

\end{proof}

In the limit case as $q\rightarrow 1$, one can also get the results
of Lemma \ref{relationlemma} for the quadratic lattice $x(s)=c_1s^2+c_2s+c_3$.

\subsection{Examples}
In this part, we construct several recurrence relations in order to show how Lemma \ref{relationlemma} works.

\begin{ex}\label{ex7}
The functions $\Psi_{\nu,\nu}$, $\Psi_{\nu,\nu-1}$ and $\Psi_{\nu,\nu-2}$ are connected by the following relation
\[
A_1(z)\Psi_{\nu,\nu}(z)+A_2(z)\Psi_{\nu,\nu-1}(z)+A_3(z)\Psi_{\nu,\nu-2}(z)=0
\]
where the coefficients $A_1(z)$, $A_2(z)$ and $A_3(z)$ are the functions in $z$

\begin{eqnarray*}
A_1(z)&=&\tau_{\nu}(0)\beta_{\nu}-\gamma_{\nu}\widetilde{\sigma}_{\nu}(0)-\tau_{\nu}(0)x(z)\\
&+&\big[\tau_{\nu}'\beta_{\nu}+\alpha_{\nu}\tau_{\nu}(0)-\gamma_{\nu}\widetilde{\sigma}_{\nu}'(0)-\tau_{\nu}'x(z)\big]x_{\nu}(z-\nu)\\
&+&\big[\tau_{\nu}'\alpha_{\nu}-\gamma_{\nu}\dsp\frac{\widetilde{\sigma}_{\nu}''}{2}\big]\big[2x_{\nu}(z-\nu)x_{\nu}(z-\nu+1)-x_{\nu}^2(z-\nu+1)\big],\\
A_2(z)&=&\tau_{\nu}'\beta_{\nu}+\alpha_{\nu}\tau_{\nu}(0)-\gamma_{\nu}\widetilde{\sigma}_{\nu}'(0)-\tau_{\nu}'x(z)\\
&+&\big[\tau_{\nu}'\alpha_{\nu}-\gamma_{\nu}\frac{\widetilde{\sigma}_{\nu}''}{2}\big]2x_{\nu}(z-\nu+1),\\
A_3(z)&=&\tau_{\nu}'\alpha_{\nu}-\gamma_{\nu}\frac{\widetilde{\sigma}_{\nu}''}{2}
\end{eqnarray*}
where $\alpha_{\nu}$, $\beta_{\nu}$ and $\gamma_{\nu}$ are defined by \refe{alphabetagammak}
and 
\begin{eqnarray}
\widetilde{\sigma}_{\nu}(s)&=&\frac{\widetilde{\sigma}_{\nu}''}{2}x_{\nu}^2(s)+\widetilde{\sigma}_{\nu}'(0)x_{\nu}(s)+\widetilde{\sigma}_{\nu}(0),\label{sigmatildenunotation}\\
\tau_{\nu}(s)&=&\tau'_{\nu}x_{\nu}(s)+\tau_{\nu}(0)\label{taununotation}
\end{eqnarray}
are the Taylor polynomial expansion of the functions $\widetilde{\sigma}_{\nu}(s)$ and $\tau_{\nu}(s)$
defined by \refe{sigmatildenu} and \refe{taunu}, respectively.
\end{ex}

\begin{proof}
By Lemma \ref{relationlemma}, we have $\nu_1=\nu$, $\nu_2=\nu$, $\nu_3=\nu$
and $\mu_1=\nu$, $\mu_2=\nu-1$, $\mu_3=\nu-2$. By formula \refe{pi}
\bq\label{pi7}
\begin{array}{ll}
\Pi(s)=&\Delta x_{\nu}(s-\frac 12)\Bigg\{A_3(z)x_{\nu}^2(s)+\Big[A_2(z)-2A_3(z)x_{\nu}(z-\nu+1)\Big]x_{\nu}(s)\\
& +\Big[A_1(z)-A_2(z)x_{\nu}(z-\nu)+A_3(z)x_{\nu}^2(z-\nu+1)\Big]\Bigg\}
\end{array}
\eq
and by \refe{piQ}
\bq\label{piQ7}
\Pi(s)=\phi_{\nu}(s)\big[x(s)-x(z)\big]Q(s+1)-\sigma(s)\big[x(s+\nu)-x(z)\big]Q(s).
\eq
Notice that if we add q-quadratic lattice \refe{q-quadraticlattice} into \refe{pi7} and use property \refe{xk}, then
$\Pi(s)$ in \refe{pi7} becomes a Laurent polynomial on the basis $\{1,z^{-1},z,z^{-2},$ $z^2,z^{-3},z^3\}$, where $z=q^s$.
Note that the L-degree of $\Pi(s)$ is six. Since the L-degree of $Q(s)$ is at least six less than $\Pi(s)$,
then the degree of $Q(s)$ becomes zero, i.e. $Q(s)=k$, where $k$ is a constant. Then, \refe{piQ7} can be rewritten as the following:

\begin{align*}
\Pi(s)&=k\Bigg\{\phi_{\nu}(s)x(s)-\sigma(s)x(s+\nu)-\big[\phi_{\nu}(s)-\sigma(s)\big]x(z)\Bigg\}\\
&=k\Bigg\{\frac 12\phi_{\nu}(s)x(s)-\frac 12\sigma(s)x(s+\nu)+\frac 12\phi_{\nu}(s)x(s)-\frac 12\sigma(s)x(s+\nu)\\
&\quad+\frac 12\phi_{\nu}(s)x(s+\nu)-\frac 12\phi_{\nu}(s)x(s+\nu)+\frac 12\sigma(s)x(s)-\frac 12\sigma(s)x(s)\\
&\quad-\big[\phi_{\nu}(s)-\sigma(s)\big]x(z)\Bigg\}.
\end{align*}
Choosing $k=1$, the above expression becomes
\begin{align*}
\Pi(s)&=\phi_{\nu}(s)\dsp\frac{x(s+\nu)+x(s)}{2}-\sigma(s)\dsp\frac{x(s+\nu)+x(s)}{2}\\
&\quad-\phi_{\nu}(s)\dsp\frac{x(s+\nu)-x(s)}{2}-\sigma(s)\dsp\frac{x(s+\nu)-x(s)}{2}\\
&\quad-\big[\phi_{\nu}(s)-\sigma(s)\big]x(z).
\end{align*}
Then, we have
\begin{align*}
\Pi(s)&=\Big[\phi_{\nu}(s)-\sigma(s)\Big]\dsp\frac{x(s+\nu)+x(s)}{2}\\
&\quad-\Big[\phi_{\nu}(s)+\sigma(s)\Big]\dsp\frac{x(s+\nu)-x(s)}{2}\\
&\quad-\big[\phi_{\nu}(s)-\sigma(s)\big]x(z).
\end{align*}
By using expressions \refe{phinu+sigma}, \refe{phinu-sigma},
\refe{latticeprop1} and \refe{latticeprop2}, we get
\[
\Pi(s)=\Delta x_{\nu}(s-\frac 12)\Bigg\{\tau_{\nu}(s)\Big[\alpha_{\nu}x_{\nu}(s)+\beta_{\nu}\Big]
-\gamma_{\nu}\widetilde{\sigma}_{\nu}(s)-\tau_{\nu}(s)x(z)\Bigg\}.
\]
Using $\widetilde{\sigma}_{\nu}(s)$ and $\tau_{\nu}(s)$ 
from \refe{sigmatildenunotation} and \refe{taununotation}, it follows
\bq\label{pi72}
\begin{array}{ll}
\Pi(s)=\Delta x_{\nu}(s-\frac 12)&\Bigg\{\Big[\tau_{\nu}'\alpha_{\nu}-\gamma_{\nu}\frac{\widetilde{\sigma}_{\nu''}}{2}\Big]
x_{\nu}^2(s)
+\Big[\tau_{\nu}'\beta_{\nu}+\alpha_{\nu}\tau_{\nu}(0)-\gamma_{\nu}\widetilde{\sigma}_{\nu'}(0)\\
&-\tau_{\nu}'x(z)\Big]x_{\nu}(s)
+\Big[\tau_{\nu}(0)\beta_{\nu}-\gamma_{\nu}\widetilde{\sigma}_{\nu}(0)-\tau_{\nu}(0)x(z)\Big]\Bigg\}.
\end{array}
\eq
Now, by equating the polynomials $\Pi(s)$
in \refe{pi7} and \refe{pi72}, we obtain the following system of equations
\begin{eqnarray*}
\!\!\!\!\!\!\!\!\!\!&\!\!\!\!\!\!\!\!\!\!&\!\!\!\!\!\!\!\!\!\!A_3(z)=\tau_{\nu}'\alpha_{\nu}-\gamma_{\nu}\frac{\widetilde{\sigma}_{\nu''}}{2}\\
\!\!\!\!\!\!\!\!\!\!&\!\!\!\!\!\!\!\!\!\!&\!\!\!\!\!\!\!\!\!\!A_2(z)-2A_3(z)x_{\nu}(z-\nu+1)=\tau_{\nu}'\beta_{\nu}+\alpha_{\nu}\tau_{\nu}(0)-\gamma_{\nu}\widetilde{\sigma}_{\nu'}(0)-\tau_{\nu}'x(z)\\
\!\!\!\!\!\!\!\!&\!\!\!\!\!\!\!\!\!\!\!\!&\!\!\!\!\!\!\!\!\!\!A_1(z)-A_2(z)x_{\nu}(z-\nu)+A_3(z)x_{\nu}^2(z-\nu+1)=\tau_{\nu}(0)\beta_{\nu}-\gamma_{\nu}\widetilde{\sigma}_{\nu}(0)-\tau_{\nu}(0)x(z).
\end{eqnarray*}
By solving this system, one can obtain the coefficients $A_1(z)$, $A_2(z)$ and $A_3(z)$.
\end{proof}

\begin{ex}\label{ex8}
The following relation holds
\[
A_1(z)\Psi_{\nu,\nu-1}(z)+A_2(z)\Psi_{\nu,\nu-2}(z)+A_3(z)\Psi_{\nu+1,\nu}(z)=0
\]
where  $A_1(z)$, $A_2(z)$ and $A_3(z)$ are the functions in $z$ 

\begin{eqnarray*}
A_1(z)&=&-\frac{\sigma(z-\nu+1)}{\nabla x_{\nu+1}(z-\nu+1)}\\
A_2(z)&=&\frac{1}{\gamma_{\nu-1}}\frac{1}{\Delta x(z)}\Bigg[\tau_{\nu}(z)-\frac{\sigma(z-\nu+1)}{\nabla x_{\nu+1}(z-\nu+1)}\Bigg]\\
A_3(z)&=&-\gamma_{\nu}.
\end{eqnarray*}
Here, $\gamma_{\nu}$ is defined by \refe{alphabetagammak}.

\end{ex}

\begin{proof}
By Lemma \ref{relationlemma}, we have $\nu_1=\nu$, $\nu_2=\nu$, $\nu_3=\nu+1$
and $\mu_1=\nu-1$, $\mu_2=\nu-2$, $\mu_3=\nu$. By formula \refe{pi},
\bq\label{pi8}
\begin{array}{lll}
\Pi(s)&=A_1(z)\nabla x_{\nu+1}(s)\big[x_{1}(s+\nu)-x_1(z)\big]\\
&+A_2(z)\nabla x_{\nu+1}(s)\big[x_{1}(s+\nu)-x_1(z)\big]\big[x_{\nu}(s)-x_{\nu}(z-\nu+1)\big]\\
&+A_3(z)\nabla x_{\nu+2}(s)\phi(s+\nu)
\end{array}
\eq
and by \refe{piQ},
\bq\label{piQ8}
\Pi(s)=\phi_{\nu}(s)\big[x_1(s)-x_1(z)\big]Q(s+1)-\sigma(s)\big[x_1(s+\nu)-x_1(z)\big]Q(s).
\eq
Notice that if we use q-quadratic lattice \refe{q-quadraticlattice} and property \refe{xk}, then 
$\Pi(s)$ in \refe{pi8} becomes a Laurent polynomial, whose basis is $\{1,z^{-1},z,z^{-2},z^2,z^{-3},z^3\}$, where $z=q^s$.
The L-degree of $\Pi(s)$ is six. Since the L-degree of $Q(s)$ is at least six less than $\Pi(s)$,
then degree of $Q(s)$ becomes zero, i.e. $Q(s)=k$, where k is a constant. Let us choose $k=1.$ 

We remark here that since $\Pi(s)$ in \refe{pi8} and \refe{piQ8} are Laurent polynomials whose basis are $\{1,z^{-1},z,z^{-2},z^2,z^{-3},z^3\}$, where $z=q^s$,
 one can find the coefficients $A_i(z)$ by equating them. Here, we consider giving particular values to make some terms of $\Pi(s)$ in \refe{pi8} or \refe{piQ8} zero.
Therefore, it becomes simpler to determine coefficients $A_i(z)$.
Firstly, let $s=z-\nu$ in $\Pi(s)$ be defined by \refe{pi8} and \refe{piQ8}.
Notice that the first two terms in \refe{pi8} and the second term of \refe{piQ8} vanish, leading to
\[
\Pi(z-\nu)=A_3(z)\phi(z)\nabla x_{\nu+2}(z-\nu)=\phi_{\nu}(z-\nu)\big[x_1(z-\nu)-x_1(z)\big],
\]
where $\phi_{\nu}(z-\nu)=\phi(z)$ by \refe{phinuphi} and
$x_1(z-\nu)-x_1(z)=-\gamma_{\nu}\nabla x_{\nu+2}(z-\nu)$ by \refe{latticeprop2} with \refe{xk}.
Then, one gets
\[
A_3(z)=-\gamma_{\nu}.
\]

In order to find $A_1(z)$ let $s=z-\nu+1$ in $\Pi(s)$ defined by \refe{pi8} and \refe{piQ8}.
Notice that the second term of \refe{pi8} vanishes and gives
\begin{eqnarray*}
\Pi(z-\nu+1)&=&A_1(z)\nabla x_{\nu+1}(z-\nu+1)\big[x_1(z+1)-x_1(z)\big]\\
&+&A_3(z)\phi(z+1)\nabla x_{\nu+2}(z-\nu+1)\\
&=&\phi_{\nu}(z-\nu+1)\big[x_1(z-\nu+1)-x_1(z)\big]\\
&-&\sigma(z-\nu+1)\big[x_1(z+1)-x_1(z)\big]
\end{eqnarray*}
where $\phi_{\nu}(z-\nu+1)=\phi(z+1)$ by \refe{phinuphi} and
$x_1(z+1)-x_1(z)=\Delta x_1(z)=\Delta x(z+\frac 12)$ by the forward operator with \refe{xk}.
Moreover, $x_1(z-\nu+1)-x_1(z)=-\gamma_{\nu}\nabla x_{\nu+2}(z-\nu+1)$ by \refe{latticeprop2}.
Replacing $A_3(z)=-\gamma_{\nu}$, one has
\[
A_1(z)=-\frac{\sigma(z-\nu+1)}{\nabla x_{\nu+1}(z-\nu+1)}.
\]

Finally, to find $A_2(z)$ we set $s=z$ in $\Pi(s)$
defined by \refe{pi8} and \refe{piQ8}.
Notice that first term of \refe{piQ8} disappears and leads to
\begin{eqnarray*}
\Pi(z)&=&A_1(z)\nabla x_{\nu+1}(z)\big[x_{1}(z+\nu)-x_1(z)\big]\\
&+&A_2(z)\nabla x_{\nu+1}(z)\big[x_{1}(z+\nu)-x_1(z)\big]\big[x_{\nu}(z)-x_{\nu}(z-\nu+1)\big]\\
&+&A_3(z)\nabla x_{\nu+2}(z)\phi(z+\nu)=-\sigma(z)\big[x_1(z+\nu)-x_1(z)\big],
\end{eqnarray*}
where $\nabla x_{\nu+2}(z)=\nabla x_{\nu}(z+1)=\Delta x_{\nu}(z)$ by \refe{forwardbackward} with \refe{xk} and
 $x_1(z+\nu)-x_1(z)=\gamma_{\nu}\Delta x_{\nu}(z)$, $x_{\nu}(z)-x_{\nu}(z-\nu+1)=\gamma_{\nu-1}\Delta x(z)$ 
by \refe{latticeprop2}, \refe{forwardbackward} with \refe{xk}.
Then, with the help of \refe{phinuphi} together with \refe{phinu-sigma} and \refe{forwardbackward}, one can have
\[
A_2(z)=\frac{1}{\gamma_{\nu-1}}\frac{1}{\Delta x(z)}\Bigg[\tau_{\nu}(z)-\frac{\sigma(z-\nu+1)}{\nabla x_{\nu+1}(z-\nu+1)}\Bigg],
\]
which completes the proof. The other proofs can be made by using the same method. Thus, we do not include them here.

\end{proof}
\begin{ex}\label{ex9}
The functions $\Psi_{\nu,\nu}$, $\Psi_{\nu,\nu-1}$ and $\Psi_{\nu+1,\nu+1}$ have the following relation
\[
A_1(z)\Psi_{\nu,\nu}(z)+A_2(z)\Psi_{\nu,\nu-1}(z)+A_3(z)\Psi_{\nu+1,\nu+1}(z)=0
\]
where coefficients $A_1(z)$, $A_2(z)$ and $A_3(z)$ are the functions in $z$

\begin{eqnarray*}
\!\!\!\!&\!\!\!\!\!\!\!\!&\!\!\!\!A_1(z)=\frac{\phi(z)}{\Delta x(z)}\Big[-\gamma_{\nu}+\gamma_{\nu+1}\frac{\nabla x_{\nu+2}(z-\nu)}{\nabla x_{\nu+1}(z-\nu)}\Big]-\frac{\sigma(z-\nu)}{\nabla x_{\nu+1}(z-\nu)}\\
\!\!\!\!&\!\!\!\!\!\!\!\!&\!\!\!\!A_2(z)=\frac{1}{\gamma_{\nu}}\frac{\tau_{\nu}(z)-A_1(z)}{\Delta x(z-\frac 12)}\\
\!\!\!\!&\!\!\!\!\!\!\!\!&\!\!\!\!A_3(z)=-\gamma_{\nu+1}
\end{eqnarray*}
where $\gamma_{\nu}$ is defined by \refe{alphabetagammak}.

\end{ex}

\begin{ex}\label{ex12}
The functions $\Psi_{\nu,\nu+1}$, $\Psi_{\nu-1,\nu}$ and $\Psi_{\nu-1,\nu-1}$ hold the relation that follows
\[
A_1(z)\Psi_{\nu,\nu+1}(z)+A_2(z)\Psi_{\nu-1,\nu}(z)+A_3(z)\Psi_{\nu-1,\nu-1}(z)=0
\]
where coefficients $A_1(z)$, $A_2(z)$ and $A_3(z)$ are the functions in $z$

\begin{eqnarray*}
A_1(z)&=&-\gamma_{\nu+1}\\
A_2(z)&=&-\gamma_{\nu}\frac{\phi(z-1)}{\Delta x(z-\frac 12)}-\frac{\sigma(z-\nu)}{\nabla x_{\nu}(z-\nu)}\\
&+&\gamma_{\nu+1}\frac{\phi(z-1)\nabla x_{\nu+1}(z-\nu)}{\Delta x(z-\frac 12)\nabla x_{\nu}(z-\nu)}\\
A_3(z)&=&\frac{1}{\gamma_{\nu}}\frac{\tau_{\nu-1}(z)-A_2(z)}{\nabla x(z)}.\\
\end{eqnarray*}
Here, $\gamma_{\nu}$ is defined by \refe{alphabetagammak}.

\end{ex}

\begin{ex}\label{ex11}
The following relation holds
\[
A_1(z)\Psi_{\nu-1,\nu-1}(z)+A_2(z)\Psi_{\nu,\nu}(z)+A_3(z)\Psi_{\nu,\nu+1}(z)=0
\]
where $A_1(z)$, $A_2(z)$ and $A_3(z)$ are the functions in $z$

\begin{eqnarray*}
\!\!\!\!&\!\!\!\!\!\!\!\!&\!\!\!\!A_1(z)=\dsp\frac{1}{\gamma_{\nu}}\Bigg\{-\frac{\sigma(z)}{\nabla x(z)}+\frac{\phi(z+\nu-1)\nabla x_{\nu}(z-\nu)}{\nabla x_{\nu}(z)\nabla x(z)}\\
&+&\gamma_{\nu}\frac{\phi(z+\nu-1)\nabla x_{\nu}(z-\nu+1)}{\nabla x_{\nu}(z)\nabla x_{\nu+1}(z-\nu)}+\frac{\phi(z+\nu-1)\sigma(z-\nu)\Delta x(z-\frac 12)}{\phi(z-1)\nabla x_{\nu}(z)\nabla x_{\nu+1}(z-\nu)}\\
&-&\gamma_{\nu+1}\frac{\phi(z+\nu-1)}{\nabla x_{\nu}(z)}\Bigg\}\\
\!\!\!\!&\!\!\!\!\!\!\!\!&\!\!\!\!A_2(z)=\gamma_{\nu+1}-\gamma_{\nu}\frac{\nabla x_{\nu}(z-\nu+1)}{\nabla x_{\nu+1}(z-\nu)}-\frac{\sigma(z-\nu)\Delta x(z-\frac 12)}{\phi(z-1)\nabla x_{\nu+1}(z-\nu)}\\
\!\!\!\!&\!\!\!\!\!\!\!\!&\!\!\!\!A_3(z)=-\gamma_{\nu+1}\nabla x_{\nu}(z-\nu)
\end{eqnarray*}
where $\gamma_{\nu}$ is defined by \refe{alphabetagammak}.

\end{ex}

\begin{ex}\label{ex10}

$\Psi_{\nu,\nu}$, $\Psi_{\nu,\nu-1}$ and $\Psi_{\nu-1,\nu-1}$ are connected by
\[
A_1(z)\Psi_{\nu,\nu}(z)+A_2(z)\Psi_{\nu,\nu-1}(z)+A_3(z)\Psi_{\nu-1,\nu-1}(z)=0
\]
where coefficients $A_1(z)$, $A_2(z)$ and $A_3(z)$ are the functions in $z$

\begin{eqnarray*}
\!\!\!\!&\!\!\!\!\!\!\!\!&\!\!\!\!A_1(z)=-\gamma_{\nu}Q(z-\nu+1)\\
\!\!\!\!&\!\!\!\!\!\!\!\!&\!\!\!\!A_2(z)=\frac{C(z)}{D(z)}\\
\!\!\!\!&\!\!\!\!\!\!\!\!&\!\!\!\!A_3(z)=-\sigma(z)+\frac{\phi(z+\nu-1)Q(z-\nu+1)}{\nabla x_{\nu}(z)}-\frac{\phi(z+\nu-1)\nabla x(z+\frac 12)}{\nabla x_{\nu}(z)}\frac{C(z)}{D(z)},
\end{eqnarray*}

where 
\begin{eqnarray*}
C(z)&=&\frac{1}{\gamma_{\nu+1}}\phi(z+\nu)\Delta x(z)\nabla x_{\nu}(z)Q(z+2)\\
&-&\sigma(z+1)\nabla x_{\nu}(z)\nabla x_{\nu}(z+1) Q(z+1)\\
&+&\frac{\gamma_{\nu}}{\gamma_{\nu+1}}\phi(z+\nu)\nabla x_{\nu}(z)\nabla x_{\nu+1}(z+1) Q(z-\nu+1)\\
&+&\sigma(z)\nabla x_{\nu}(z+1)\nabla x_{\nu}(z+1) Q(z)\\
&-&\phi(z+\nu-1)\nabla x_{\nu}(z+1)\nabla x_{\nu}(z+1) Q(z-\nu+1),
\end{eqnarray*}
\begin{eqnarray*}
D(z)&=&\phi(z+\nu)\nabla x_{\nu}(z)\nabla x_{\nu+1}(z+1)\nabla x(z+1)\\
&-&\phi(z+\nu-1)\nabla x_{\nu}(z+1)\nabla x_{\nu}(z+1)\nabla x(z+\frac 12)
\end{eqnarray*}
and $\gamma_{\nu}$ is defined by \refe{alphabetagammak}. Here, $Q(z)$ is a first-degree polynomial in $x(z).$ In particular, considering $Q(z)=\nabla x_{\nu}(z)$ leads to the following formulas:

\begin{eqnarray*}
\!\!\!\!&\!\!\!\!\!\!\!\!&\!\!\!\!A_1(z)=-\gamma_{\nu}\nabla x_{\nu}(z-\nu+1)\\
\!\!\!\!&\!\!\!\!\!\!\!\!&\!\!\!\!A_2(z)=\frac{C(z)}{D(z)}\\
\!\!\!\!&\!\!\!\!\!\!\!\!&\!\!\!\!A_3(z)=-\sigma(z)+\frac{\phi(z+\nu-1)\nabla x_{\nu}(z-\nu+1)}{\nabla x_{\nu}(z)}-\frac{\phi(z+\nu-1)\nabla x(z+\frac 12)}{\nabla x_{\nu}(z)}\frac{C(z)}{D(z)},
\end{eqnarray*}
where 
\begin{eqnarray*}
C(z)&=&\frac{1}{\gamma_{\nu+1}}\phi(z+\nu)\Delta x(z)\nabla x_{\nu}(z)\nabla x_{\nu}(z+2)\\
&-&\sigma(z+1)\nabla x_{\nu}(z)\nabla x_{\nu}(z+1)\nabla x_{\nu}(z+1)\\
&+&\frac{\gamma_{\nu}}{\gamma_{\nu+1}}\phi(z+\nu)\nabla x_{\nu}(z)\nabla x_{\nu+1}(z+1)\nabla x_{\nu}(z-\nu+1)\\
&+&\sigma(z)\nabla x_{\nu}(z+1)\nabla x_{\nu}(z+1)\nabla x_{\nu}(z)\\
&-&\phi(z+\nu-1)\nabla x_{\nu}(z+1)\nabla x_{\nu}(z+1)\nabla x_{\nu}(z-\nu+1),
\end{eqnarray*}
\begin{eqnarray*}
D(z)&=&\phi(z+\nu)\nabla x_{\nu}(z)\nabla x_{\nu+1}(z+1)\nabla x(z+1)\\
&-&\phi(z+\nu-1)\nabla x_{\nu}(z+1)\nabla x_{\nu}(z+1)\nabla x(z+\frac 12).
\end{eqnarray*}

\end{ex}
Next section includes the recurrence relations for the solutions of the second-order linear difference equation of hypergeometric type \ref{diffeqn} using Lemma \ref{relationlemma}.

\section{Recurrence relations of the solutions of the second-order linear difference equation of hypergeometric type}
We first remark that the solution $y_{\nu}$ of the difference equation \ref{diffeqn} can be rewritten using the function $\Psi_{\nu,\mu}$ as
\bq\label{ynuz}
y_{\nu}(z)=\frac{C_{\nu}}{\rho(z)}\Psi_{\nu,\nu}(z).
\eq 
Here, we include the recurrence relations related with solutions $y_{\nu}$
and their difference derivatives as defined in \cite{sus1, sus2} by
\bq\label{ynuk}
y_{\nu}^{(k)}:=\Delta^{(k)}y_{\nu}(s)=\frac{C_{\nu}^{(k)}}{\rho_k(s)}\Psi_{\nu,\nu-k}(s)
\eq
where
\[
\Delta^{(k)}=\Bigg(\frac{\Delta}{\Delta x_{k-1}}\Bigg)...\Bigg(\frac{\Delta}{\Delta x_{1}}\Bigg)\Bigg(\frac{\Delta}{\Delta x_{0}}\Bigg),
\]
\bq\label{rhok}
\rho_k(s)=\sigma(s+1)\rho_{k-1}(s+1)=\rho(s+k)\prod_{i=1}^k\sigma(s+i),
\eq
\[
C_{\nu}^{(k)}=\Bigg[\alpha_{\nu-k}\widetilde{\tau}'_{k-1}+\gamma_{\nu-k}\frac{\widetilde{\sigma}''_{k-1}}{2}\Bigg]C_{\nu}^{(k-1)}=\kappa_{\nu+k-1}C_{\nu}^{(k-1)}=\prod_{i=0}^{k-1}\kappa_{\nu+i}C_{\nu},
\]
\[
\kappa_{\nu}=\alpha_{\nu-1}\widetilde{\tau}'+\gamma_{\nu-1}\frac{\widetilde{\sigma}''}{2}, \qquad C_{\nu}^{(0)}=C_{\nu}.
\]
The following theorem has been proved for the lineer-type lattices in \cite{ran3, RC1, RC2}. 
It is also valid for the quadratic-type lattices but together with the condition of Lemma \ref{relationlemma}.

\begin{theorem}\label{derivativerelationlemma}
The following linear relation holds
\bq\label{derivativerelation}
\sum_{i=1}^3B_i(s)y_{\nu_i}^{(k_i)}(s)=0,
\eq
by the conditions of Lemma \ref{relationlemma},
where
$$B_i(s)=A_i(s)(C_{\nu_i}^{(k_i)})^{-1}\phi(s+k_*)...\phi(s+k_i-1).$$
Here, $A_i(s)$ are the coefficient functions of the recurrence relations defined in Lemma \ref{relationlemma}.
\end{theorem}
\begin{proof}
By Lemma \ref{relationlemma}, there exist the functions $A_i(z)$, $i=1,2,3$
such that the following linear relation holds
\[
\sum_{i=1}^3A_i(s)\Psi_{\nu_i, \nu_i-k_i}(s)=0.
\]
Therefore, by the definition of the difference derivative \refe{ynuk}, we have
\[
\sum_{i=1}^3A_i(s)(C_{\nu_i}^{(k_i)})^{-1}\rho_{k_i}(s)y_{\nu_i}^{(k_i)}(s)=0,
\]
which can be rewritten as the following
\[
\sum_{i=1}^3B_i(s)y_{\nu_i}^{(k_i)}(s)=0, \quad B_i(s)=A_i(s)(C_{\nu_i}^{(k_i)})^{-1}\phi(s+k_*)...\phi(s+k_i-1),
\]
by dividing the equality with $\rho_{k_*}(s)$, where $k_*=\min\{k_1,k_2,k_3\}$ and then, using expression \refe{rhonu*},
which completes the proof.
\end{proof}

By the formula defined in \refe{ynuk}, the examples \ref{ex7}, \ref{ex8}, \ref{ex9} and \ref{ex10} 
lead to the following relations
\[
\begin{array}{llll}
&B_1(s)y_{\nu}(s)+B_2(s)y^{(1)}_{\nu}(s)+B_3(s)y^{(2)}_{\nu}(s)=0,\\
&B_1(s)y^{(1)}_{\nu}(s)+B_2(s)y^{(2)}_{\nu}(s)+B_3(s)y^{(1)}_{\nu+1}(s)=0,\\
&B_1(s)y_{\nu}(s)+B_2(s)y^{(1)}_{\nu}(s)+B_3(s)y_{\nu+1}(s)=0,\\
&B_1(s)y_{\nu}(s)+B_2(s)y^{(1)}_{\nu}(s)+B_3(s)y_{\nu-1}(s)=0.
\end{array}
\]
Notice that the last two relations are the so-called raising and lowering operators, respectively,
which are equivalent to the following $\Delta$-ladder-type recurrence relations, respectively,
	\bq\label{deltaladder+1}
	B_1(s)y_{\nu}(s)+B_2(s)\frac{\Delta y_{\nu}(s)}{\Delta x(s)}+B_3(s)y_{\nu+1}(s)=0,
	\eq
	\bq\label{deltaladder-1}
	\widetilde{B}_1(s)y_{\nu}(s)+\widetilde{B}_2(s)\frac{\Delta y_{\nu}(s)}{\Delta x(s)}+\widetilde{B}_3(s)y_{\nu-1}(s)=0.
	\eq

	Note that in order to get a $\Delta$-ladder-type recurrence relation, it is sufficient to put $k_1=0$, $k_2=1$, $k_3=0$ and $\nu_1=\nu$, $\nu_2=\nu$, $\nu_3=\nu+m$ into \refe{derivativerelation}, where $m=\mp1$.

\begin{corollary}

The following three-term recurrence relation holds
\[
B_1(s)y_{\nu}(s)+B_2(s)y_{\nu+1}(s)+B_3(s)y_{\nu-1}(s)=0,
\]
provided that the conditions in Lemma \ref{relationlemma} exist.
Here, coefficients $B_i(s)$, $i=1,2,3$ are the polynomial functions.
\end{corollary}
\begin{proof}
By substituting $k_1=0$, $k_2=0$, $k_3=0$ and $\nu_1=\nu$, $\nu_2=\nu+1$, $\nu_3=\nu-1$
in \refe{derivativerelation}, one can obtain the above relation.
\end{proof}

%%%%%%%%%%%%%%%%%%%%%%%%%%%%%%%%%%%%%%%%%%%%%%%%%%%%%%%%%%%%%%%%%%%%%%%%%%%%%%

\section{Applications to polynomials on the quadratic-type lattices}

In this section, we include the applications of the method to the $q$-Racah and dual Hahn polynomials
which are defined by \refe{sumynu} with $\nu=n$. 
These polynomials are general and defined
on the $q$-quadratic lattices of the form $x(s) = q^{-s}+\delta \gamma q^{s+1}$
and the quadratic lattices of the form $x(s)=s(s+1)$, respectively.

One can find a detailed study on these polynomials in \cite{ran1, ks, NSU}.
Since the $q$-Racah and dual Hahn polynomials
are defined by \refe{sumynu} where $\nu=n$,
then condition \refe{sumynucondition} is satisfied. Therefore,
Lemma \ref{relationlemma} and Theorem \ref{derivativerelationlemma} hold for such polynomials.

In the following, we include the two types of recurrence relations consisting of the polynomials and their difference-derivatives.

\subsection{The application of the method to the $q$-Racah polynomials}
Let $\nu_1=n$, $\nu_2=n$, $\nu_3=n+1$ and $k_1=0$, $k_2=1$, $k_3=0$
in Theorem \ref{derivativerelationlemma}, then we get
\[
B_1(s)P_{n}(s)+B_2(s)\Delta^{(1)}P_{n}(s)+B_3(s)P_{n+1}(s)=0.
\]

\begin{eqnarray*}
\!\!\!\!&\!\!\!\!\!\!\!\!&\!\!\!\!B_1(s)=\frac{1}{C_n}\left\{\frac{\phi(s)}{\Delta x(s)}\Big[-\gamma_{n}+\gamma_{n+1}\frac{\nabla x_{n+2}(s-n)}{\nabla x_{n+1}(s-n)}\Big]-\frac{\sigma(s-n)}{\nabla x_{n+1}(s-n)}\right\}\\
\!\!\!\!&\!\!\!\!\!\!\!\!&\!\!\!\!B_2(s)=\frac{\phi(s)}{C_n^{(1)}}\frac{1}{\gamma_{n}}\frac{\tau_{n}(s)-C_nB_1(s)}{\Delta x(s-\frac 12)}\\
\!\!\!\!&\!\!\!\!\!\!\!\!&\!\!\!\!B_3(s)=-\frac{\gamma_{n+1}}{C_{n+1}}
\end{eqnarray*}
where $\gamma_{n}$ is defined by \refe{alphabetagammak} and $C_n^{(1)}=(\alpha_{n-1}\widetilde{\tau}'+\gamma_{n-1}\frac{\widetilde{\sigma}''}{2})C_n$. 
Notice that this case is considered in example 8 with $n=\nu$.
Placing the corresponding values of the $q$-Racah polynomials from Table 1
in coefficients $B_i(s)$, $i=1,2,3$, we have
\begin{eqnarray*}
B_1(s)&=&\frac{1}{C_n}\Big\{\frac{-q^{1/2}(q^{1/2}-q^{-1/2})q^s(q^{-s}-\alpha q)(q^{-s}-\gamma q)(1-\beta\delta q^{s+1})(1-\delta\gamma q^{s+1})}{(1-\delta\gamma q^{2s+2})}\\
&\times&\Big[-[n]_q+[n+1]_q\frac{q^{-1/2}(1-\delta\gamma q^{2s-n+2})}{(1-\delta\gamma q^{2s-n+1})}\Big]\\
&+&\frac{q^{2-n/2}(q^{1/2}-q^{-1/2})q^s(1-\delta q^{s-n})(1-q^{-s+n})(\alpha-\delta\gamma q^{s-n})(\gamma-\beta\gamma q^{-s+n})}{(1-\delta\gamma q^{2s-n+1})}\Big\},\\
B_2(s)&=&\frac{2(q^{1/2}-q^{-1/2})(q^{-s}-\alpha q)(q^{-s}-\gamma q)(1-\beta\delta q^{s+1})(1-\gamma\delta q^{s+1})}{C_n[n]_q\Big[(q^{(n-1)/2}-q^{-(n-1)/2})(1-\beta\gamma q^{2})+[n-1]_qq^{-n}(q^{1/2}-q^{-1/2})(1+\alpha\beta q^{2n+2})\Big]}\\
&\times&\frac{-(1-\alpha\beta q^{2n+2})x_n(s)+q^{-n/2}[(1-\alpha q^{n+1})(1-\beta\delta q^{n+1})(1-\gamma q^{n+1})\Big\}}{q^{-s}(1-\delta\gamma q^{2s+2})}\\
&-&\frac{(q^{1/2}-q^{-1/2})\Big\{(1-\delta\gamma q^{n+1})(1-\alpha \beta q^{2n+2})]\Big\}+C_nB_1(s)}{(q^{1/2}-q^{-1/2})q^{-s}(1-\delta\gamma q^{2s+2})},\\
B_3(s)&=&-\frac{[n+1]_q}{C_{n+1}}.
\end{eqnarray*}

\subsection{The application of the method to the dual Hahn polynomials}

Let $\nu_1=n-1$, $\nu_2=n$, $\nu_3=n+1$ and $k_1=0$, $k_2=1$, $k_3=0$
in Theorem \ref{derivativerelationlemma}, then we get
\[
B_1(s)P_{n-1}(s)+B_2(s)\Delta^{(1)}P_{n}(s)+B_3(s)P_{n+1}(s)=0.
\]

In order to obtain this relation, we use the following three-term recurrence relation (TTRR )
\begin{equation}\label{TTRR}
x(s)P_n(s)=\widetilde{\alpha}_nP_{n+1}(s)+\widetilde{\beta}_nP_n(s)+\widetilde{\gamma}_nP_{n-1}(s), \quad n=0, 1, 2, ...,
\end{equation}
with the initial conditions $P_0(s)=1,  P_{-1}(s)=0$, and
also the differentiation formula \cite[Eq. (5.67)]{ran1} (or
\cite[Eq. (25)]{RYR})
\begin{equation}\label{diff.formula2}
\phi(s)\frac{\Delta P_n(s)_q}{\Delta x(s)}=\widehat{\alpha}_n P_{n+1}(s)_q
+\widehat{\beta}_n(s)P_n(s)_q,
\end{equation}
where $\phi(s)=\sigma(s)+\tau(s)\Delta x(s-\frac{1}{2})$, and
$$
\widehat{\alpha}_n= -\frac{ \lambda_{n}}{[n]_q\widetilde{\tau}'_n}\frac{B_n}{B_{n+1}},\quad
\widehat{\beta}_n(s)=\frac{\lambda_{n}}{[n]_q}\frac{\tau_n(s)}{\widetilde{\tau}'_n}-\lambda_n \Delta x(s-\half).
$$
Notice that the above differentiation formula is valid for the $q$-polynomials on the $q$-quadratic lattices. 
In order to obtain a formula  for the polynomials on the quadratic lattices,
one can consider the limit case when $q\rightarrow 1$.

Now, to compute $\Delta^{(1)}P_{n}(s)=\frac{\Delta P_n(s)}{\Delta x(s)}$, we first multiply the above equality by $\phi(s)$ and then use formula
\refe{diff.formula2}, then we reach
\[
\!\!\!\!\!\!\!\!\!\!\!B_1(s)\phi(s)P_{n-1}(s)+B_2(s)\left[\widehat{\alpha}_{n}P_{n+1}(s)+\widehat{\beta}_{n}(s)P_n(s)\right]
+B_3(s)\phi(s)P_{n+1}(s)=0,
\]
which can be rewritten as the following form
\[
\!\!\!\!\!\!\!\!\!\!\!\left[B_2(s)\widehat{\alpha}_{n}+B_3(s)\phi(s)\right]P_{n+1}(s)
+B_2(s)\widehat{\beta}_n(s)P_{n}(s)
+B_1(s)\phi(s)P_{n-1}(s)=0.
\]
By the TTRR, we have the following system of equations
\[
B_2(s)\widehat{\alpha}_{n}+B_3(s)\phi(s)=\widetilde{\alpha}_n, \quad
B_2(s)\widehat{\beta}_n(s)=\widetilde{\beta}_n-x(s),\quad
B_1(s)\phi(s)=\widetilde{\gamma}_n,
\]
which leads to
\[
B_1(s)=\dsp\frac{\widetilde{\gamma}_n}{\phi(s)}, \,
B_2(s)=\dsp\frac{\widetilde{\beta}_n-x(s)}{\widehat{\beta}_{n}(s)}, \,
B_3(s)=\dsp\frac{1}{\phi(s)}\left[\widetilde{\alpha}_{n}-\frac{\widehat{\alpha}_{n}}{\widehat{\beta}_{n}(s)}(\widetilde{\beta}_n-x(s))\right].
\]
Considering the limit case as $q\rightarrow 1$, the above coefficients become 
\[
B_1(s)=\dsp\frac{\gamma_n}{\phi(s)}, \,
B_2(s)=\dsp\frac{\beta_n-x(s)}{\widehat{\beta}_{n}(s)}, \,
B_3(s)=\dsp\frac{1}{\phi(s)}\left[\alpha_{n}-\frac{\widehat{\alpha}_{n}}{\widehat{\beta}_{n}(s)}(\beta_n-x(s))\right].
\]
Then, inserting the corresponding values of the dual Hahn polynomials from Table \ref{dualhahntable}, \cite[Table 3.7., Page 109]{NSU}
in coefficients $B_i(s)$, $i=1,2,3$ leads to
\begin{eqnarray*}
\!\!\!\!\!\!\!\!\!&\!\!\!\!\!\!\!\!\!&\!\!\!\!\!\!\!\!\!B_1(s)=\frac{(a+c+n)(b-a-n)(b-c-n)}{(s+a+1)(s+c+1)(b-s-1)},\\
\!\!\!\!\!\!\!\!\!&\!\!\!\!\!\!\!\!\!&\!\!\!\!\!\!\!\!\!B_2(s)=-\frac{[ab-ac+bc+(b-a-c-1)(2n+1)-2n^2-s(s+1)](2s+n+1)}{\kappa_{n+1}},\\
\!\!\!\!\!\!\!\!\!&\!\!\!\!\!\!\!\!\!&\!\!\!\!\!\!\!\!\!B_3(s)=\frac{n+1}{(s+a+1)(s+c+1)(b-s-1)}\left[1+B_2(s)\right],
\end{eqnarray*}
where $\kappa_n=(s+a+n)(s+c+n)(b-s-n)-(s-a)(s+b)(s-c)+(n-1)(2s+1)(2s+n).$

\begin{table}[ht!]
\caption{Main data of the monic $q$-Racah polynomials}
\label{racahtable}
\begin{center}
\tiny
{\renewcommand{\arraystretch}{.25}
\begin{tabular}{|@{}c@{}| | @{}c@{}|}\hline
 &  \\
$P_n(s)$ & $R_n(x(s),\alpha,\beta,\gamma,\delta|q)   \,,\quad x(s) = q^{-s}+\delta \gamma q^{s+1}$ , 
\,\,\, $\Delta x(s)=q^{-s}(1-\delta\gamma q^{2s+2})(q^{-1}-1)$ \\
 &  \\
\hline\hline
 &  \\
$\sigma(s)$ & \mbox{
$  \delta^2\gamma^2q^2(q^{1/2}-q^{-1/2})^2q^{-2s}(q^s-1)(q^s-\delta^{-1})
(q^s-\beta \gamma^{-1})(q^s-\alpha\delta^{-1}\gamma^{-1}) $} \\
 &  \\
\hline
 &  \\
$\phi(s)$ & \mbox{
$  (q^{1/2}-q^{-1/2})^2q^{-2s}(1-\alpha q^{s+1})(1-\beta\delta q^{s+1})(1-\gamma q^{s+1})(1-\delta \gamma q^{s+1}) $} \\
 &  \\
\hline
 &  \\
$\tau(s) $ &  \mbox{
\begin{tabular}{r}
$ \frac{(q^{-1/2}-q^{1/2})q^{-s}}{(1-\gamma\delta q^{2s+1})}\Big[(1-\alpha q^{s+1})(1-\beta\delta q^{s+1})(1-\gamma q^{s+1})(1-\delta \gamma q^{s+1})-(\delta\gamma q)^2 (q^s-1)$  \\[3mm]
 $\times (q^s-\delta^{-1})(q^s-\beta \gamma^{-1})(q^s-\alpha\delta^{-1}\gamma^{-1})\Big]$ \end{tabular}} \\
 &  \\
 \hline
 &  \\
$\tau'$ &\mbox{
$(q^{1/2}-q^{-1/2})(1-\beta\gamma q^2)$} \\
 &  \\
\hline
$\tau_n(s) $ &  \mbox{%\scriptsize
\begin{tabular}{r}
$ -(q^{1/2}-q^{-1/2})\Big\{(1-\alpha\beta q^{2n+2})x(s+\frac n2)+q^{-n/2}\Big[(1-\alpha q^{n+1})
(1-\beta\delta q^{n+1})(1-\gamma q^{n+1})
$\\[3mm] $-(1+\delta \gamma q^{n+1})(1-\alpha\beta q^{2n+2})\Big]\Big\}$ \end{tabular}}\\
 &  \\
 \hline
 &  \\
$\lambda _n$ &\mbox{
$-q^{-n+\textstyle \frac 12}(1-q^n)(1-\alpha\beta q^{n+1})$} \\
 &  \\
\hline
 &  \\
$\frac{\widetilde{\sigma}_{\nu}''}{2}$ &    \mbox{$\frac 12q^{-n}(q^{1/2}-q^{-1/2})^2(1+\alpha\beta q^{2n+2})
$}\,\,\, \\
 &  \\
 \hline
 &  \\
$\widetilde{\sigma}_{\nu}'(0)$ &    \mbox{$-\frac 12q^{-\frac n2}(q^{1/2}-q^{-1/2})^2\Bigg[\beta\delta q+\alpha q+\delta \gamma q+\gamma q+\alpha\beta\delta q^{n+2}+\alpha\beta q^{n+2}+\beta \delta \gamma q^{n+2}+\alpha \gamma q^{n+2}\Bigg]
$}\,\,\, \\
 &  \\
 \hline
 &  \\
$\widetilde{\sigma}_{\nu}(0)$ &    \mbox{$\delta \gamma qq^{-n}(q^{1/2}-q^{-1/2})^2\Bigg[\beta\alpha q^{n+1}+\beta\delta q^{n+1}+\alpha q^{n+1}+\beta q^{n+1}+\alpha\delta^{-1} q^{n+1}+\gamma q-1-\alpha\beta q^{2n+2}\Bigg]
$}\,\,\, \\
 &  \\
 \hline

 $\widetilde{\beta}_n $ &  \begin{tabular}{c}
$\dst1+\delta \gamma q-\frac{(1-\alpha q^{n+1})(1-\alpha\beta q^{n+1})(1-\beta\delta q^{n+1})(1-\gamma q^{n+1})}
{(1-\alpha\beta q^{2n+1})(1-\alpha\beta q^{2n+2})}$ \\[3mm]
$\dst-\frac{q(1-q^n)(1-\beta q^n)(\gamma-\alpha\beta q^n)(\delta-\alpha q^n)}
{(1-\alpha\beta q^{2n})(1-\alpha\beta q^{2n+1})}$
\end{tabular}\\
 &  \\
\hline
 &  \\
$\widetilde{\gamma}_n$ & $ \dst\frac{(1-\alpha q^{n})(1-\alpha\beta q^{n})(1-\beta\delta q^{n})(1-\gamma q^{n})}
{(1-\alpha\beta q^{2n-1})(1-\alpha\beta q^{2n})}$ 
$\dst\frac{q(1-q^n)(1-\beta q^n)(\gamma-\alpha\beta q^n)(\delta-\alpha q^n)}
{(1-\alpha\beta q^{2n})(1-\alpha\beta q^{2n+1})} $ \\
 &  \\
\hline
&  \\
$\widehat{\alpha}_n$ &  $q^{-n+\textstyle\frac 12}(q^{-1/2}-q^{1/2})(1-\alpha\beta q^{2n+1})$  \\&  \\
\hline &  \\
$\overline{\beta}_n(s) $ & $\begin{tabular}{r}
$-q^{-n/2+1/2}(q^{1/2}-q^{-1/2})\dst\frac{(1-\alpha\beta q^{n+1})}{(1-\alpha\beta q^{2n+2})}\Big\{q^{-1}(1-\alpha\beta q^{2n+2})x(s+\frac n2)+q^{-n/2}q^{-n-1}\Big[(1-\alpha q^{n+1})
$\\[3mm] $\times (1-\beta\delta q^{n+1})(1-\gamma q^{n+1})-(1+\delta \gamma q^{n+1})(1-\alpha\beta q^{2n+2})\Big]\Big\}$ \end{tabular}$\\
\hline &\\
$\widehat{\beta}_n(s) $ & 
$\overline{\beta}_n(s)-q^{-s-n+\textstyle\frac 12}(q^{1/2}-q^{-1/2})(1-q^n)(1-\alpha\beta q^{n+1})(1-\delta \gamma q^{2s+1})$ \\
\hline
\end{tabular}  }
\end{center}
\end{table}

\begin{table}[ht!]
\caption{Main data of the dual Hahn polynomials} 
\label{dualhahntable}
\begin{center}
{\renewcommand{\arraystretch}{.25}
\begin{tabular}{|@{}c@{}| | @{}c@{}|}\hline
 &  \\
$P_n(s)$ & $W^{(c)}_n(x(s))   \,,\quad x(s) =s(s+1)$ , 
\,\,\, $\Delta x(s)=2s+2$ \\
 &  \\
\hline\hline
 &  \\
$\sigma(s)$ & \mbox{
$(s-a)(s+b)(s-c)$} \\
 &  \\
\hline
 &  \\
$\phi(s)$ & \mbox{
$  (s+a+1)(s+c+1)(b-s-1) $} \\
 &  \\
\hline
 &  \\
$\tau(s) $ &  \mbox{
$ ab-ac+bc-a+b-c-1-x(s)$}   \\
 &  \\
 \hline
 &  \\
$\lambda _n$ &\mbox{
$n$} \\
 &  \\
\hline
 &  \\
 $\alpha_n=\widehat{\alpha}_n $ & \mbox{
$n+1$}
\\
 &  \\
\hline
 &  \\
 $\beta_n $ & \mbox{
$ab-ac+bc+(b-a-c-1)(2n+1)-2n^2$}
\\
 &  \\
\hline
 &  \\
$\gamma_n$ & \mbox{$(a+c+n)(b-a-n)(b-c-n) $} \\
 &  \\
\hline
&  \\
$\widehat{\beta}_n(s) $ &\mbox{
$\frac{(s+a+n+1)(s+c+n+1)(b-s-n-1)-(s-a)(s+b)(s-c)+n(2s+1)(2s+n+1)}{2s+n+1}$} \\
\hline
\end{tabular}}
\end{center}
\end{table}

%%%%%%%%%%%%%%%%%%%%%%%%%%%%%%%%%%%%%%%%%%555555555555555555555555555555555
\section{Concluding remarks}
In the present work,  some recurrence relations are developed 
for the hypergeometric functions on the quadratic-type lattices with
applications in the $q$-Racah and dual Hahn polynomials. 
To obtain the recurrence relations for the other classes of polynomials, one can use appropriate limit transitions.

\section*{Acknowledgements:} 
The author would like to thank to Prof. R. \'{A}lvarez-Nodarse for his valuable comments and suggestions in preparing this manuscript.

%%%%%%%%%%%%%%%%%%%%%%%%%%%%%%%%%%%%%%%%%


\bigskip
\begin{thebibliography}{00}

\bibitem{ran1}  R. \'{A}lvarez-Nodarse,
{\it  Polinomios hipergem\'etricos y q-polinomios,} Monograf\'{\i}as del
Seminario Garc\'{\i}a Galdeano. Universidad de Zaragoza.  Vol. {\bf 26}.
Prensas Universitarias de Zaragoza, Zaragoza, Spain, 2003. (In Spanish).

\bibitem{ran2}  R. \'{A}lvarez-Nodarse,
{\it  On characterizations of classical polynomials}, J. Comput. Appl. Math. {\bf 196}.
(2006), pp. 320-337.


\bibitem{RAC}  R. \'{A}lvarez-Nodarse, N. M. Atakishiyev and R. S. Costas-Santos,
{\it  Factorization of hypergeometric type difference equations on nonuniform lattices: dynamical algebra,} J. Phys. A: Math. Gen.  {\bf 38} (2005), pp. 153-174.

\bibitem{ran3}  R. \'{A}lvarez-Nodarse,
{\it  \'{A} La Carte recurrence relations for continuous and discrete hypergeometric functions}, Sema Journal {\bf 55}.
(2011), pp. 41-57.


\bibitem{RC1}  R. \'{A}lvarez-Nodarse, J. L. Cardoso,
{\it  Recurrence relations for discrete hypergeometric functions}, J. Difference Eq. Appl. {\bf 11}.
(2005), pp. 829-850.

\bibitem{RC2}  R. \'{A}lvarez-Nodarse, J. L. Cardoso,
{\it  On the Properties of Special Functions on the linear-type lattices}, J. Math. Anal. Appl. {\bf 405}.
(2013), pp. 271-285.

\bibitem{RCQ}  R. \'{A}lvarez-Nodarse, J. L. Cardoso and N. R. Quintero,
{\it  On recurrence relations for radial wave functions for the N-th dimensional oscillators and hydrogenlike atoms: analytical and numerical study}, Elect. Trans. Num. Anal. {\bf 24}.
(2006), pp. 7-23.

\bibitem{RM}  R. \'{A}lvarez-Nodarse, J. C. Medem,
{\it  $q$-Classical polynomials and the $q$-Askey and Nikiforov-Uvarov Tableaus}, J. Comput. Appl. Math. {\bf 135}.
(2001), pp. 157-196.



\bibitem{RR1}  R. \'Alvarez-Nodarse, R. Sevinik Ad\i g\"{u}zel,
{On the Krall type polynomials on q-quadratic lattices,}
{\it Indagationes Mathematicae}, {\bf 21} (2011), pp. 181-203,
doi:10.1016/j.indag.2011.04.002.


\bibitem{RR2}  R. \'Alvarez-Nodarse, R. Sevinik Ad\i g\"{u}zel,
{The $q$-Racah-Krall-type polynomials,}
{\it Appl. Math. Comput.}, {\bf 218} (2012), pp. 11362-11369.


\bibitem{RYR}  R. \'Alvarez-Nodarse, Yu. F. Smirnov, R. S. Costas-Santos,
{A $q$-Analog of Racah Polynomials and q-Algebra $SU_q(2)$ in Quantum Optics,}
{\it J.Russian Laser Research} {\bf 27} {(2006)}, 1-32.


\bibitem{ARS}  N. M. Atakishiyev, M. Rahman, S. K. Suslov,
{Classical Orthogonal Polynomials,}
{\it Constr. Approx. } {\bf 11} {(1995)}, pp. 181-226.


\bibitem{GR}  G. Gasper and M. Rahman
{Basic Hypergeometric Series (2nd Ed.),}
{\it Encyclopedia of Mathematics and its Applications } {\bf 96}, Cambridge University Press, Cambridge, 2004.

\bibitem{GST1}  A. Gill, J. Segura, and N.M. Temme, 
\textit{The ABC of hyper recursions,}
J. Comput. Appl. Math., {\bf 190} {(2006)}, pp. 270-286.

\bibitem{GST2}  A. Gill, J. Segura, and N.M. Temme, 
\textit{Numerically Satisfactory solutions of hypergeometric recursions,}
Math. Comp., {\bf 76} {(2007)}, pp. 1449-1468.

\bibitem{JN}  W.B. Jones, O. Nj\*{a}stad, 
\textit{Orthogonal Laurent Polynomials and strong moment theory: a survey,}
J. Comput. Appl. Math., {\bf 105} {(1999)}, pp. 51-91.


\bibitem{ks}  R. Koekoek, Peter A. Lesky, and R.F.  Swarttouw, 
\textit{Hypergeometric orthogonal polynomials and their q-analogues,}
Springer Monographs in Mathematics, Springer-Verlag, Berlin-Heidelberg, 2010.


\bibitem{NSU} A. F. Nikiforov, S. K. Suslov, and V. B. Uvarov,
{\em Classical Orthogonal Polynomials of a Discrete Variable,}
{Springer Ser. Comput. Phys., Springer-Verlag, Berlin, 1991.}


\bibitem{sus1}  S. K. Suslov,
{\it  The theory of difference analogues of special functions of hypergeometric type  }, Russian Math. Surveys {\bf 44}:2
(1989), pp. 227-278.


\bibitem{sus2}  S. K. Suslov,
{\it  A correction}, Russian Math. Surveys {\bf 45}:3
(1990), pp. 245-245.

\end{thebibliography}
\end{document}